\documentclass[12pt,a4paper]{article}
\usepackage{amssymb}
\usepackage{amsmath}
\usepackage{latexsym}
\usepackage{amsthm}

\newtheorem{thm}{Theorem}
\newtheorem{lem}[thm]{Lemma}

\usepackage{xpatch}
\xpatchcmd{\proof}{\itshape}{\normalfont\proofnameformat}{}{}
\newcommand{\proofnameformat}{}

\begin{document}

\renewcommand{\proofnameformat}{\bfseries}

\begin{center}
{\large\textbf{Empirical measures and random walks on compact spaces in the quadratic Wasserstein metric}}

\vspace{5mm}

\textbf{Bence Borda}

{\footnotesize Graz University of Technology

Steyrergasse 30, 8010 Graz, Austria

Email: \texttt{borda@math.tugraz.at}}

\vspace{5mm}

{\footnotesize \textbf{Keywords:} Riemannian manifold, Lie group, optimal transportation,\\ Berry--Esseen inequality, heat kernel, occupation measure}

\end{center}

\vspace{5mm}

\begin{abstract}
Estimating the rate of convergence of the empirical measure of an i.i.d.\ sample to the reference measure is a classical problem in probability theory. Extending recent results of Ambrosio, Stra and Trevisan on 2-dimensional manifolds, in this paper we prove sharp asymptotic and nonasymptotic upper bounds for the mean rate in the quadratic Wasserstein metric $W_2$ on a $d$-dimensional compact Riemannian manifold. Under a smoothness assumption on the reference measure, our bounds match the classical rate in the optimal matching problem on the unit cube due to Ajtai, Koml\'os, Tusn\'ady and Talagrand. The i.i.d.\ condition is relaxed to stationary samples with a mixing condition. As an example of a nonstationary sample, we also consider the empirical measure of a random walk on a compact Lie group. Surprisingly, on semisimple groups random walks attain almost optimal rates even without a spectral gap assumption. The proofs are based on Fourier analysis, and in particular on a Berry--Esseen smoothing inequality for $W_2$ on compact manifolds, a result of independent interest with a wide range of applications.
\end{abstract}

\section{Introduction}

Given a sequence of random variables $X_1, X_2, \dots$, each with distribution $\mu$, the empirical measure $\mu_N = N^{-1}\sum_{n=1}^N \delta_{X_n}$ converges weakly to $\mu$ under very general circumstances. Estimating the rate of convergence in the Wasserstein metric $W_p$ has received considerable attention. Classical results of Ajtai, Koml\'os, Tusn\'ady \cite{AKT} and Talagrand \cite{TA} on optimal matchings concern i.i.d.\ random variables uniformly distributed on the unit cube $[0,1]^d$, in which case\footnote{Throughout, $a_N \ll b_N$ and $a_N=O(b_N)$ mean that there exists an implied constant $C>0$ and $N_0 \in \mathbb{N}$ such that $|a_N| \le C b_N$ for all $N \ge N_0$. Similar notation is used for functions.}
\begin{equation}\label{AKTresult}
\mathbb{E} W_1 (\mu_N, \mu ) \ll \left\{ \begin{array}{ll} N^{-1/2} & \textrm{if } d=1, \\ (\log N)^{1/2} N^{-1/2} & \textrm{if } d=2, \\ \sqrt{d} N^{-1/d} & \textrm{if } d \ge 3 \end{array} \right.
\end{equation}
with a universal implied constant, and this is sharp for all $d \ge 1$. Several deep results on $W_p(\mu_N, \mu)$ have since appeared for more general probability measures $\mu$ on Euclidean spaces, including moment estimates, concentration inequalities and central limit theorems, see \cite{BL2} for a comprehensive account of the $1$-dimensional case. The i.i.d.\ condition can also be relaxed: the empirical measure of stationary $m$-dependent or $\rho$-mixing sequences, and that of Markov chains have been shown to behave similarly to the i.i.d.\ case under suitable conditions \cite{DM,FG}.

Weak convergence of empirical measures remains perfectly meaningful on more general metric spaces, with the Wasserstein metric providing a natural way to quantify the rate of convergence. Boissard and Le Gouic \cite{BLG} showed that an i.i.d.\ sample with an arbitrary distribution $\mu$ taking values in a compact metric space satisfies the sharp estimate
\begin{equation}\label{d/2>presult}
\mathbb{E} W_p (\mu_N, \mu ) \ll N^{-1/d} \qquad \textrm{provided that } 1 \le p < d/2,
\end{equation}
and that the i.i.d.\ condition can be relaxed to suitable stationary $\rho$-mixing sequences and to Markov chains. Here $d>0$ is the ``dimension'', in the sense that for all small $R>0$, the metric space can be covered by $\ll R^{-d}$ balls of radius $R$. See also \cite{BW,BOI}.

In this paper, we study the convergence rate of empirical measures on a compact, connected, smooth Riemannian manifold (without boundary) $M$ of dimension $d$, with Riemannian volume $\mathrm{Vol}$. We shall only use the quadratic Wasserstein metric $W_2$, defined in terms of the geodesic distance $\rho$ on $M$. The case of a uniformly distributed (i.e.\ $\mu = \mathrm{Vol}/\mathrm{Vol}(M)$) i.i.d.\ sample in dimension $2$ has recently been settled by Ambrosio, Stra and Trevisan \cite{AST}, who showed the remarkable asymptotic relation
\begin{equation}\label{d=2result}
\sqrt{\mathbb{E}W_2^2 (\mu_N, \mu )} \sim \sqrt{\frac{\mathrm{Vol}(M)}{4 \pi}} \cdot \sqrt{\frac{\log N}{N}} \qquad \textrm{as } N \to \infty .
\end{equation}
One of the goals of the present paper is to extend this result to higher dimensions, and to relax the i.i.d.\ condition. For the latter purpose, we shall use the pairwise mixing coefficients $\alpha (X_m, X_n)$ and $\beta (X_m,X_n)$ (see Section \ref{notationsection} for definitions).

To state our results, let $X_1, X_2, \dots$ be a sequence of $M$-valued random variables, each with distribution $\mu$ (a Borel probability measure on $M$), and let $\mu_N=N^{-1} \sum_{n=1}^N \delta_{X_n}$.
\begin{thm}\label{alphamixingtheorem} If $N^{-2} \sum_{1 \le m<n \le N} \alpha (X_m, X_n) \to 0$, then $\mathbb{E}W_2^2 (\mu_N, \mu) \to 0$ as $N \to \infty$.
\end{thm}
\noindent The following nonasymptotic result concerns a finite sample $X_1, X_2, \dots, X_N$ under a smoothness assumption on $\mu$.
\begin{thm}\label{betamixingtheorem} Assume that $\mu \ge c \mathrm{Vol}/\mathrm{Vol}(M)$ and that $\sum_{1 \le m<n \le N} \beta (X_m, X_n) \le B N$ with some constants $c>0$ and $B \ge 0$.
\begin{enumerate}
\item[(i)] If $d=2$, then
\[ \sqrt{\mathbb{E}W_2^2(\mu_N, \mu)} \le \sqrt{\frac{\mathrm{Vol}(M) + B C(M)}{\pi c}} \cdot \sqrt{\frac{\log N}{N}} + \frac{C(M)}{\sqrt{c N}} \]
with some constant $C(M)>0$ depending only on the manifold.
\item[(ii)] If $d \ge 3$, then
\[ \sqrt{\mathbb{E}W_2^2(\mu_N, \mu)} \le \left( \frac{\mathrm{Vol}(M) + B C(M)}{c} \right)^{1/d} \frac{\kappa \sqrt{d}}{N^{1/d}} + C(M) \left( \frac{1+B}{cN} \right)^{3/(2d)} \]
with
\[ \kappa = \frac{1}{\sqrt{\pi}} \left( 1+(1-c)^{1/2} \right)^{1-2/d} \left( \frac{8}{d(d-2)} \right)^{1/d} \le \frac{2}{\sqrt{\pi}} \]
and some constant $C(M)>0$ depending only on the manifold.
\end{enumerate}
If there exists an orthonormal basis $\{ \phi_k \, : \, k \ge 0 \}$ of $L^2(M, \mathrm{Vol}/\mathrm{Vol}(M))$ of eigenfunctions of the Laplace--Beltrami operator such that $\sup_{k \ge 0} \sup_{x \in M} |\phi_k(x)| < \infty$, then the same holds with $\beta (X_m, X_n)$ replaced by $\alpha (X_m, X_n)$.
\end{thm}
\noindent The constants $C(M)$ depend only on simple geometric and spectral properties of the manifold $M$. In fact, we shall derive Theorems \ref{alphamixingtheorem} and \ref{betamixingtheorem} in a completely explicit form, see Theorem \ref{weaklydependenttheorem} in Section \ref{weaklydependentsection}. The additional condition on the existence of an orthonormal basis of Laplace eigenfunctions with bounded sup-norms is satisfied by the flat torus $\mathbb{R}^d/\mathbb{Z}^d$ and certain other flat manifolds \cite{TZ}, and is an open problem e.g.\ on the $2$-dimensional unit sphere \cite{VK}.

In addition to i.i.d.\ samples, Theorem \ref{betamixingtheorem} applies just as well to pairwise independent samples with $B=0$. For a general survey and the history of various mixing coefficients we refer to Bradley \cite{BR}. The notion of the $\alpha$-mixing (also known as strong mixing) coefficient goes back to Rosenblatt, while the $\beta$-mixing coefficient (also known as coefficient of absolute regularity) was first introduced in \cite{VR1,VR2}, where it was attributed to Kolmogorov. Although we do not use it in this paper, we also mention the perhaps better known $\phi$-mixing coefficient. These are related by the general inequalities $2 \alpha (X,Y) \le \beta (X,Y) \le \phi (X,Y)$, hence our results apply in particular to $\phi$-mixing samples. We refer to \cite{BR,MP} and references therein for quantitative bounds for the $\beta$-mixing coefficients of Markov chains on a general state space. Every strictly stationary, aperiodic, Harris recurrent Markov chain $X_1, X_2, \ldots$ satisfies $\beta (X_m,X_n) \to 0$ as $|m-n| \to \infty$. If the strictly stationary Markov chain is geometrically ergodic, then the convergence $\beta (X_m,X_n) \to 0$ is at least exponentially fast, whereas subgeometric ergodicity leads to slower convergence rates. (Sub)geometric ergodicity in turn is ensured by suitable drift conditions. Quantitative bounds for the $\beta$-mixing coefficients have also been established in the setting of SDEs \cite{MA,VE}.

In the case of a uniformly distributed, pairwise independent sample in dimension $2$, Theorem \ref{betamixingtheorem} (i) with $c=1$ and $B=0$ recovers the upper bound part of \eqref{d=2result} up to a factor of $2$. Theorem \ref{betamixingtheorem} (ii) seems to be new even in the i.i.d.\ case in dimensions $3$ and $4$, and is comparable to \eqref{d/2>presult} in dimension $d \ge 5$. The only $1$-dimensional compact, connected Riemannian manifold is the unit circle; we include it for the sake of completeness.
\begin{thm}\label{circletheorem} Let $M=\mathbb{R}/\mathbb{Z}$, normalized so that $\mathrm{Vol}(\mathbb{R}/\mathbb{Z})=1$. Assume that $\mu \ge c \mathrm{Vol}$ and that $\sum_{1 \le m<n \le N} \alpha (X_m, X_n) \le B N$ with some constants $c>0$ and $B \ge 0$. Then
\[ \sqrt{\mathbb{E}W_2^2 (\mu_N, \mu)} \le \sqrt{\frac{2+16B}{3cN}} . \]
\end{thm}

The smoothness assumption on $\mu$ cannot be removed from Theorems \ref{betamixingtheorem} and \ref{circletheorem} in dimensions $d=1,2,3$ even in the i.i.d.\ case, but \eqref{d/2>presult} suggests that it might be superfluous in dimension $d \ge 5$ under a suitable mixing condition. Dimension $d=4$ (when $p=d/2$) seems the most delicate case in this regard. The necessity of the smoothness assumption in low dimensions follows from an observation originally made for Euclidean spaces in \cite{FG}, but which applies equally well on compact metric spaces, and which we now recall. Let $A,B \subseteq M$ be two Borel sets with $\mathrm{dist}(A,B)>0$. Let $\mu$ be a Borel probability measure on $M$ such that $\mu (A \cup B)=1$ and $\mu (A), \mu(B)>0$. If the sample is i.i.d.\ with distribution $\mu$, then the indicators $\xi_n=I_{\{ X_n \in A \}}$, $1 \le n \le N$ are i.i.d.\ Bernoulli variables with $\mathbb{E}\xi_n = \mu (A)$. It is not difficult to see that
\[ W_2^2 (\mu_N , \mu ) \ge \mathrm{dist} (A,B)^2 \left| \frac{1}{N} \sum_{n=1}^N (\xi_n - \mathbb{E} \xi_n) \right| , \]
since at least $|\mu_N(A)-\mu(A)|=|N^{-1} \sum_{n=1}^N (\xi_n - \mathbb{E} \xi_n)|$ amount of mass has to be transported either from $A$ to $B$, or from $B$ to $A$. Consequently $\sqrt{\mathbb{E} W_2^2 (\mu_N, \mu)} \gg N^{-1/4}$, thus Theorems \ref{betamixingtheorem} and \ref{circletheorem} do not apply to $\mu$ in dimensions $d=1,2,3$. Note that the measure $\mu$ constructed above can be e.g.\ absolutely continuous with bounded density; the problem comes instead from the fact that its support is not connected.

It would be interesting to extend Theorem \ref{betamixingtheorem} to higher moments, concentration inequalities and almost sure asymptotics of $W_2(\mu_N,\mu)$ with weakly dependent samples. It seems likely that such results will hold only under stronger mixing assumptions.

Our approach is a simplified version of that in \cite{AST}. In a nutshell, the Benamou--Brenier formula relates the linearization of the quadratic Wasserstein metric to a suitable negative Sobolev norm $\dot{H}_{-1}$, and the latter has a natural Fourier analytic interpretation. This leads to a Berry--Esseen type smoothing inequality for $W_2$ on compact Riemannian manifolds, a result we believe to be of independent interest, see Theorem \ref{smoothingtheorem} in Section \ref{smoothingsection}. Estimating the convergence rate of empirical measures is, as we will see, one of many applications of such a smoothing inequality. For further interplay between Fourier analysis and the Wasserstein metric, see \cite{BOR1,BST,CA,STE1,STE2}.

The same approach of using a smoothing inequality to estimate the rate of convergence of empirical measures has recently been used by Bobkov and Ledoux \cite{BL1,BL3} on the flat torus $\mathbb{R}^d/\mathbb{Z}^d$, who showed that the classical estimate \eqref{AKTresult} for $\mathbb{E}W_1 (\mu_N,\mu )$ remains true for an identically distributed sample with arbitrary distribution $\mu$ under the mixing condition $\sum_{1 \le m<n \le N} \alpha (X_m, X_n) \ll N$. Our results show that $\sqrt{\mathbb{E} W_2^2 (\mu_N, \mu )}$ satisfies the same classical estimate, under an additional smoothness assumption on $\mu$.

Our results can be extended to nonstationary processes, such as Markov chains on the state space $M$. In this paper we work out the details in the special case of random walks on compact groups. Let $G$ be a compact, connected Lie group of dimension $d$ equipped with an invariant Riemannian metric, normalized so that $\mathrm{Vol}(G)=1$. Note that the Riemannian volume $\mathrm{Vol}$ is now the Haar measure. Let $Y_1, Y_2, \dots$ be a sequence of i.i.d.\ $G$-valued random variables, each with distribution $\nu$, and let $S_n=Y_1 Y_2 \cdots Y_n$ be the corresponding random walk, whose distribution is thus the $n$-fold convolution power $\nu^{*n}$. A classical theorem of Kawada and It\^o \cite{KI,STR} states that $\nu^{*n} \to \mathrm{Vol}$ weakly as $n \to \infty$ if and only if the support of $\nu$ is contained neither in a proper closed subgroup, nor in a coset of a proper closed normal subgroup of $G$.

It is not surprising that under a spectral gap condition, $W_2 (\nu^{*n}, \mathrm{Vol}) \to 0$ exponentially fast, and $\sqrt{\mathbb{E}W_2^2(\mu_N, \mathrm{Vol})}$ with the empirical measure $\mu_N=N^{-1} \sum_{n=1}^N \delta_{S_n}$ satisfies the classical bound \eqref{AKTresult}. See Theorems \ref{qnutheorem} and \ref{rwempiricaltheorem} in Section \ref{spectralgapsection} for an explicit and more general form of this fact. We mention that explicit constructions of discrete $\nu$ with a spectral gap are known on $\mathrm{SU}(n)$ and other simple groups \cite{BS,BG1,BG2}. In fact, the so-called spectral gap conjecture, a deep unsolved problem in the theory of compact groups, predicts that on a semisimple, compact, connected Lie group $G$ (such as $\mathrm{SU}(n)$, $n \ge 2$ or $\mathrm{SO}(n)$, $n \ge 3$), a probability measure $\nu$ has a spectral gap whenever $\nu^{*n} \to \mathrm{Vol}$ weakly. Recall that every such group is of dimension $d \ge 3$. The spectral gap conjecture would thus imply the remarkable fact that the purely qualitative (and trivially necessary) Kawada--It\^o condition automatically improves to optimal rates for the random walk $S_n$. Using the best known partial result on the spectral gap conjecture due to Varj\'u \cite{VA}, we show that this remarkable self-improving property holds unconditionally in a somewhat weaker form.

\newpage

\begin{thm}\label{semisimpletheorem} Let $G$ be a semisimple, compact, connected Lie group of dimension $d$, and let $\nu$ be a Borel probability measure on $G$. If $\nu^{*n} \to \mathrm{Vol}$ weakly as $n \to \infty$, then
\[ W_2 (\nu^{*n}, \mathrm{Vol}) \ll e^{-an^{1/3}} \quad \textrm{and} \quad \sqrt{\mathbb{E}W_2^2 (\mu_N, \mathrm{Vol})} \ll \frac{(\log N)^{2/d}}{N^{1/d}} \]
with some constant $a>0$ and implied constants depending only on $\nu$ and $G$.
\end{thm}
The bound on the empirical rate is new, whereas the almost exponential bound on the rate of weak convergence $\nu^{*n} \to \mathrm{Vol}$ improves our recent result in \cite{BOR1} from $W_1$ to $W_2$. The condition of semisimplicity cannot be removed. Indeed, in a recent paper \cite{BOR2} we constructed certain discrete random walks on the flat torus $\mathbb{R}^d/\mathbb{Z}^d$, and found the precise rate of convergence of $W_p (\nu^{*n}, \mathrm{Vol})$, $0<p \le 1$ to be polynomial instead of (almost) exponential. We mention that using the methods of this paper, these results can also be improved to $0<p \le 2$, and in particular \cite[Theorem 6]{BOR2} remains true verbatim with $W_1$ replaced by $W_2$. We refer to \cite{BOR2} also for the closely related problem of additive functionals of the process $S_n$, such as the functional central limit theorem and the functional law of the iterated logarithm for the sum $\sum_{n=1}^N f(S_n)$ with a H\"older continuous function $f:G \to \mathbb{R}$.

Finally, we comment on the problem of quantization of measure, where the goal is to approximate a given probability measure by a finitely supported one. For the sake of simplicity, we only consider the Haar measure on a compact, connected Lie group $G$. It follows from a standard ball packing argument that for any $1 \le p < \infty$,
\begin{equation}\label{quantization}
\frac{1}{N^{1/d}} \ll \inf_{|\mathrm{supp} (\nu)| \le N} W_p (\nu, \mathrm{Vol} ) \ll \frac{1}{N^{1/d}} ,
\end{equation}
where the infimum is over all Borel probability measures $\nu$ on $G$ supported on at most $N$ points (with arbitrary weights). See \cite{GL,KL} for more precise and general results. The main message of the theory of empirical measures is that random point sets, such as weakly dependent random variables or random walks, attain this optimum in dimension $d \ge 3$, but not in dimensions $d=1,2$.

It is also an important problem to explicitly construct (deterministic) point sets attaining the optimum in \eqref{quantization}. As an application of our smoothing inequality, in Section \ref{spectralgapsection} we deduce a sharp upper bound for $W_2(\nu, \mathrm{Vol})$ in terms of the spectral radius of the Markov operator associated to $\nu$. In particular, we will show that if $\nu$ is supported on at most $N$ points and its Markov operator has spectral radius $\ll N^{-1/2}$, then $\nu$ attains the optimal distance $W_2 (\nu, \mathrm{Vol}) \ll N^{-1/d}$ from the uniform distribution. Finitely supported measures with equal weights having optimally small spectral radius were first constructed by Lubotzky, Phillips and Sarnak \cite{LPS1,LPS2} on $\mathrm{SU}(2)$ and $\mathrm{SO}(3)$; by our results, these consequently also attain the optimum in \eqref{quantization} for $p=2$. See \cite{CL,OH} for similar constructions of finite point sets on more general spaces.

\section{Notation}\label{notationsection}

Throughout the paper, $M$ is a compact, connected, smooth Riemannian manifold without boundary. Let $d$ denote its dimension, $\mathrm{Vol}$ the Riemannian volume, $\rho$ the geodesic distance, and $\Delta$ the Laplace--Beltrami operator. Let $0=\lambda_0<\lambda_1 \le \lambda_2 \le \cdots$ be the spectrum of $M$, each value repeated according to its multiplicity, and let $\phi_k$, $k \ge 0$ be the corresponding (real-valued) smooth eigenfunctions; that is, $\Delta \phi_k =-\lambda_k \phi_k$. We normalize the eigenfunctions so that $\{ \phi_k \, : \, k \ge 0 \}$ is an orthonormal basis in $L^2(M, \mathrm{Vol}/\mathrm{Vol}(M))$. Let
\[ P(t,x,y) = \sum_{k=0}^{\infty} e^{-\lambda_k t} \phi_k (x) \phi_k (y) , \qquad t>0, \,\,\, x,y \in M \]
denote the heat kernel. The series is uniformly convergent on $[t_0,\infty) \times M \times M$ for any $t_0>0$, and is normalized as
\[ \frac{1}{\mathrm{Vol}(M)} \int_M P(t,x,y) \, \mathrm{d} \mathrm{Vol} (x) = \frac{1}{\mathrm{Vol}(M)} \int_M P(t,x,y) \, \mathrm{d} \mathrm{Vol} (y) = 1. \]
Further, $P(t,x,y)>0$, it is $C^1$ in the variable $t>0$, smooth in $(x,y) \in M \times M$, and satisfies the heat equation $\frac{\partial}{\partial t} P(t,x,y) = \Delta_x P(t,x,y) = \Delta_y P(t,x,y)$; here $\Delta_x$ resp.\ $\Delta_y$ denote the Laplacian with respect to the variable $x$ resp.\ $y$.

Let $\mathcal{P}(M)$ denote the set of Borel probability measures on $M$. For any $f \in L^1 (M, \mathrm{Vol})$ and $\mu \in \mathcal{P}(M)$, let $\widehat{f}(k)=\mathrm{Vol}(M)^{-1} \int_M f \phi_k \, \mathrm{d} \mathrm{Vol}$ and $\widehat{\mu}(k)=\int_M \phi_k \, \mathrm{d} \mu$ denote the Fourier coefficients.

The Wasserstein metric $W_p$, $0<p<\infty$ is defined as
\[ W_p (\mu, \nu ) = \inf_{\pi \in \mathrm{Coup}(\mu, \nu)} \left( \int_{M \times M} \rho (x,y)^p \, \mathrm{d} \pi (x,y) \right)^{\min \{ 1/p, 1 \}} , \qquad \mu, \nu \in \mathcal{P}(M), \]
where $\mathrm{Coup}(\mu, \nu )$ is the set of couplings, i.e.\ the set of all $\pi \in \mathcal{P}(M \times M)$ whose marginals are $\pi (B \times M) = \mu (B)$ and $\pi (M \times B) = \nu (B)$, $B \subseteq M$ Borel. For any $0<p<\infty$, $W_p$ is a metric on $\mathcal{P}(M)$ and it metrizes weak convergence.

The $\alpha$-mixing coefficient of two $M$-valued random variables $X$ and $Y$ is defined as
\[ \begin{split} \alpha (X,Y) &= \sup_{A,B} |\Pr (X \in A, \,\, Y \in B) - \Pr (X \in A) \Pr (Y \in B)| \\ &= \sup_{0 \le f,g \le 1} |\mathbb{E} f(X) g(Y) - \mathbb{E} f(X) \mathbb{E}g(Y)|, \end{split} \]
where the supremum is over all Borel sets $A,B \subseteq M$ resp.\ all Borel measurable functions $f,g : M \to [0,1]$. The $\beta$-mixing coefficient is defined as
\[ \beta (X,Y) = \frac{1}{2} \sup_{I, J \in \mathbb{N}} \sup_{\substack{A_1, \dots, A_I \\ B_1, \dots, B_J}} \sum_{i=1}^I \sum_{j=1}^J |\Pr (X \in A_i, \,\, Y \in B_j) - \Pr (X \in A_i) \Pr (Y \in B_j)| , \]
where the supremum is over all partitions $A_1, \dots, A_I$ and $B_1, \dots, B_J$ of $M$ into Borel sets. Letting $\vartheta \in \mathcal{P}(M \times M)$, $\mu , \nu \in \mathcal{P}(M)$ denote the distributions of $(X,Y)$, $X$, $Y$, respectively, we can also express it as the total variation distance
\[ \beta (X,Y) = \frac{1}{2} \| \vartheta - \mu \otimes \nu \|_{\mathrm{TV}} = \sup_{C} |\vartheta (C) - (\mu \otimes \nu)(C)| = \sup_{0 \le F \le 1} \left| \int_{M \times M} F \, \mathrm{d} (\vartheta - \mu \otimes \nu ) \right| ,\]
the supremum being over all Borel sets $C \subseteq M \times M$ resp.\ all Borel measurable functions $F: M \times M \to [0,1]$. In comparison, we mention that the $\phi$-mixing coefficient is defined as
\[ \phi (X,Y) = \sup_{\substack{A,B \\ \Pr (X \in A) >0 }} \left| \Pr \left( Y \in B \mid X \in A \right) - \Pr (Y \in B) \right| . \]

For a general reference on the various topics appearing in this paper, see Villani \cite{VI} on the Wasserstein metric and its connection to optimal transportation, Aubin \cite{AU} on analysis and Sobolev spaces on Riemannian manifolds, Chavel \cite{CH} on the Laplacian and heat kernel on Riemannian manifolds, and Bourbaki \cite{BOU} on compact Lie groups.

\section{A smoothing inequality for $W_2$}\label{smoothingsection}

The main tool in the present paper is the following Berry--Esseen type inequality, whose proof is based on a smoothing procedure using the heat kernel. It generalizes and makes explicit a result in \cite{BST}, where the case $\mu=\mathrm{Vol}/\mathrm{Vol}(M)$ and $\nu$ a finitely supported probability measure with equal weights was considered. We have recently proved a similar Berry--Esseen inequality for $W_1$ on compact Lie groups without any smoothness assumption on the measures \cite{BOR1}.
\begin{thm}\label{smoothingtheorem} Let $\mu, \nu \in \mathcal{P}(M)$, and assume that $\mu \ge c \mathrm{Vol}/\mathrm{Vol} (M)$ with some constant $c \ge 0$. For any real $t>0$,
\[ W_2 (\mu, \nu ) \le c_1 (\mu ) \left( dt+K(M)t^{3/2} \right)^{1/2} + \frac{2}{c_2(\mu ,t)} \left( \sum_{k=1}^{\infty} \frac{e^{-\lambda_k t}}{\lambda_k} |\widehat{\mu}(k)-\widehat{\nu}(k)|^2 \right)^{1/2} \]
with
\begin{equation}\label{c1c2def}
c_1(\mu ) = 1+(1-c)^{1/2}, \qquad c_2(\mu, t) = \left( \inf_{y \in M} \int_M P(t/2,x,y) \, \mathrm{d} \mu (x) \right)^{1/2} \ge c^{1/2}
\end{equation}
and some constant $K(M)$ depending only on the manifold. If the Ricci curvature of $M$ is $\ge -A$ with some constant $A \ge 0$, then
\begin{equation}\label{K(M)}
K(M)=\frac{2(d-1)\sqrt{A}}{3} \left( d+(d-1) \sqrt{A} \, \mathrm{diam} \, M \right)^{1/2}
\end{equation}
is a suitable choice.
\end{thm}
\noindent In particular, $K(M)=0$ whenever the Ricci curvature is positive semidefinite, which is the case e.g.\ for the $d$-dimensional unit sphere and for any compact, connected Lie group equipped with an invariant Riemannian metric. Note that if $c=0$ (i.e.\ without any smoothness assumption on $\mu$), we still have $c_2(\mu, t) \ge \left( \inf_{x,y \in M} P(t/2,x,y) \right)^{1/2}>0$, but we might not have a positive lower bound independent of $t$.

\subsection{Dispersion rate}

We now prove an estimate on the dispersion rate, and a simple fact about Sobolev norms.
\begin{lem}\label{dispersionlemma} For any real $t>0$ and any $x \in M$,
\[ \frac{1}{\mathrm{Vol}(M)} \int_M P(t,x,y) \rho(x,y)^2 \, \mathrm{d} \mathrm{Vol} (y) \le 2dt + K(M) (2t)^{3/2} \]
with some constant $K(M)$ depending only on the manifold. If the Ricci curvature of $M$ is $\ge -A$ with some constant $A \ge 0$, then \eqref{K(M)} is a suitable choice.
\end{lem}

\begin{proof} Fix $x \in M$, and consider the function
\[ F(t) := \frac{1}{\mathrm{Vol}(M)} \int_M P(t,x,y) \rho(x,y)^2 \, \mathrm{d} \mathrm{Vol} (y) , \qquad t>0 . \]
Clearly, $\lim_{t \to 0^+} F(t) = \rho (x,x)^2 =0$, and Green's identity shows that
\[ \begin{split} F'(t) = &\frac{1}{\mathrm{Vol}(M)} \int_M \frac{\partial}{\partial t} P(t,x,y) \rho (x,y)^2 \, \mathrm{d} \mathrm{Vol} (y) \\ = &\frac{1}{\mathrm{Vol}(M)} \int_M (\Delta_y P(t,x,y)) \rho (x,y)^2 \, \mathrm{d} \mathrm{Vol} (y) \\ = &\frac{1}{\mathrm{Vol}(M)} \int_M P(t,x,y) (\Delta_y \rho (x,y)^2) \, \mathrm{d} \mathrm{Vol} (y) . \end{split} \]
By the Laplacian comparison theorem, here
\begin{equation}\label{laplaciancomparison}
\Delta_y \rho (x,y)^2 \le 2d+2(d-1) \sqrt{A} \rho (x,y) ,
\end{equation}
and we deduce
\begin{equation}\label{F'(t)upperbound}
F'(t) \le 2d + \frac{2(d-1) \sqrt{A}}{\mathrm{Vol}(M)} \int_M P(t,x,y) \rho(x,y) \, \mathrm{d} \mathrm{Vol} (y) .
\end{equation}
To be more precise, since $\rho (x,y)^2$ is not smooth at the cut locus of the point $x$, we used global Laplacian comparison, i.e.\ the fact that \eqref{laplaciancomparison} remains true on all of $M$ in the sense of distributions \cite{WE}.

The trivial bound $\rho (x,y) \le \mathrm{diam} \, M$ in \eqref{F'(t)upperbound} gives $F'(t) \le 2d + 2(d-1) \sqrt{A} \, \mathrm{diam} \, M$, and by integrating we immediately deduce the preliminary estimate
\[ F(t) \le (2d + 2(d-1) \sqrt{A} \, \mathrm{diam} \, M) t . \]
On the other hand, applying Cauchy--Schwarz in \eqref{F'(t)upperbound},
\[ \begin{split} F'(t) &\le 2d + 2(d-1) \sqrt{A} \sqrt{F(t)} \\ &\le 2d + 2(d-1) \sqrt{A} \left( 2d + 2(d-1) \sqrt{A} \, \mathrm{diam} \, M \right)^{1/2} t^{1/2} . \end{split} \]
By integrating, we deduce
\[ F(t) \le 2dt + \frac{4(d-1) \sqrt{A}}{3} \left( 2d + 2(d-1) \sqrt{A} \, \mathrm{diam} \, M \right)^{1/2} t^{3/2}, \]
as claimed.
\end{proof}

\begin{lem}\label{sobolevlemma} If $f \in L^2(M, \mathrm{Vol})$ is differentiable in the sense of distributions, then
\[ \sum_{k=1}^{\infty} \lambda_k |\widehat{f}(k)|^2 \le \frac{1}{\mathrm{Vol}(M)} \int_M |\nabla f|^2 \, \mathrm{d} \mathrm{Vol} . \]
\end{lem}

\begin{proof} In the special case $f \in C^2 (M)$, the claim easily follows from Green's identity and the Parseval formula:
\[ \frac{1}{\mathrm{Vol}(M)} \int_M |\nabla f|^2 \, \mathrm{d} \mathrm{Vol} = \frac{1}{\mathrm{Vol}(M)} \int_M (-\Delta f) f \, \mathrm{d} \mathrm{Vol} = \sum_{k=1}^{\infty} \lambda_k |\widehat{f}(k)|^2 . \]
Recall that $C^{\infty}(M)$ is dense in the Sobolev space of all $f \in L^2(M, \mathrm{Vol})$ differentiable in the sense of distributions such that $\nabla f \in L^2 (M, \mathrm{Vol})$, with the Sobolev norm
\[ \left( \frac{1}{\mathrm{Vol} (M)} \int_M |f|^2 \, \mathrm{d} \mathrm{Vol} \right)^{1/2} + \left( \frac{1}{\mathrm{Vol} (M)} \int_M |\nabla f|^2 \, \mathrm{d} \mathrm{Vol} \right)^{1/2} . \]
Now let $f \in L^2(M, \mathrm{Vol})$ be differentiable in the sense of distributions; we may assume that $\nabla f \in L^2(M, \mathrm{Vol})$. We thus have a sequence $f_n \in C^{\infty}(M)$ converging to $f$ in the Sobolev norm above. In particular,
\[ \sum_{k=1}^{\infty} \lambda_k |\widehat{f_n}(k)|^2 = \frac{1}{\mathrm{Vol}(M)} \int_M |\nabla f_n|^2 \, \mathrm{d} \mathrm{Vol} \to \frac{1}{\mathrm{Vol}(M)} \int_M |\nabla f|^2 \, \mathrm{d} \mathrm{Vol} \quad \textrm{as } n \to \infty . \]
Since $f_n \to f$ in $L^2(M, \mathrm{Vol})$, we also have $\widehat{f_n}(k) \to \widehat{f}(k)$ as $n \to \infty$ for any fixed $k \ge 1$, and the claim follows.
\end{proof}

\subsection{Proof of Theorem \ref{smoothingtheorem}}

Given $\mu \in \mathcal{P}(M)$, let $\dot{H}_1(\mu)$ be the set of all functions $f \in L^2(M, \mathrm{Vol})$ differentiable in the sense of distributions such that
\[ \| f \|_{\dot{H}_1(\mu )} := \left( \int_{M} |\nabla f|^2 \, \mathrm{d} \mu \right)^{1/2} < \infty . \]
For any signed Borel measure $\vartheta$ on $M$, let
\[ \| \vartheta \|_{\dot{H}_{-1}(\mu)} := \sup \left\{ \left| \int_M f \, \mathrm{d} \vartheta \right| \, : \, f \in \dot{H}_1(\mu ), \| f \|_{\dot{H}_1(\mu)} \le 1 \right\} . \]
The proof of Theorem \ref{smoothingtheorem} relies on a result of Peyre \cite{PEY}, who showed that for any $\mu, \nu \in \mathcal{P}(M)$,
\begin{equation}\label{peyre}
W_2 (\mu, \nu ) \le 2 \| \mu - \nu \|_{\dot{H}_{-1}(\mu)}.
\end{equation}
His argument is based on the Benamou--Brenier formula
\[ W_2 (\mu, \nu ) = \inf_{\gamma} \int_0^1 \| \mathrm{d} \gamma (t) \|_{\dot{H}_{-1}(\gamma (t))} , \]
where the infimum is over suitable curves $\gamma : [0,1] \to \mathcal{P}(M)$ with $\gamma (0)=\mu$ and $\gamma (1)=\nu$; choosing $\gamma (t)=(1-t)\mu+t \nu$ gives \eqref{peyre}.

\begin{proof}[Proof of Theorem \ref{smoothingtheorem}] Convolving the measures $\mu$ and $\nu$ with the heat kernel leads to the smoothed measures
\[ \begin{split} \mu_t (B) &:= \frac{1}{\mathrm{Vol}(M)} \int_M \int_M P(t,x,y) I_B(y) \, \mathrm{d} \mu(x) \mathrm{d} \mathrm{Vol} (y) \qquad (B \subseteq M \,\, \textrm{Borel}), \\ \nu_t (B) &:= \frac{1}{\mathrm{Vol}(M)} \int_M \int_M P(t,x,y) I_B(y) \, \mathrm{d} \nu(x) \mathrm{d} \mathrm{Vol} (y) \qquad (B \subseteq M \,\, \textrm{Borel}) . \end{split} \]
The triangle inequality for $W_2$ gives
\begin{equation}\label{triangleW2}
W_2 (\mu, \nu) \le W_2 (\mu, \mu_t) + W_2 (\mu_t, \nu_t) + W_2 (\nu_t, \nu) .
\end{equation}

To estimate the last term in \eqref{triangleW2}, consider $\pi \in \mathcal{P} (M \times M)$,
\[ \pi (C) := \frac{1}{\mathrm{Vol}(M)} \int_M \int_M P(t,x,y) I_C(x,y) \, \mathrm{d} \nu(x) \mathrm{d} \mathrm{Vol}(y) \qquad (C \subseteq M \times M \,\, \textrm{Borel}) . \]
Observe that the marginals of $\pi$ are $\pi (B \times M) = \nu(B)$ and $\pi (M \times B) = \nu_t (B)$. We can bound the cost of the transport plan $\pi$ using Lemma \ref{dispersionlemma}:
\[ \begin{split} W_2(\nu, \nu_t) &\le \left( \int_{M \times M} \rho(x,y)^2 \, \mathrm{d} \pi (x,y) \right)^{1/2} \\ &= \left( \frac{1}{\mathrm{Vol}(M)} \int_M \int_M P(t,x,y) \rho (x,y)^2 \, \mathrm{d} \nu (x) \mathrm{d} \mathrm{Vol} (y) \right)^{1/2} \\ &\le \sup_{x \in M} \left( \frac{1}{\mathrm{Vol}(M)} \int_M P(t,x,y) \rho (x,y)^2 \, \mathrm{d} \mathrm{Vol} (y) \right)^{1/2} \\ &\le \left( 2dt + K(M) (2t)^{3/2} \right)^{1/2} . \end{split} \]
Estimating the first term in \eqref{triangleW2} is analogous. The only difference is that by the assumption $\mu \ge c \mathrm{Vol} / \mathrm{Vol}(M)$, we can leave $c$ amount of mass intact and construct a similar transport plan only for the remaining $(1-c)$ amount of mass. We thus obtain
\[ W_2(\mu, \mu_t) \le (1-c)^{1/2} \left( 2dt + K(M) (2t)^{3/2} \right)^{1/2} . \]

Finally, let us apply Peyre's estimate \eqref{peyre} to the second term in \eqref{triangleW2}:
\[ W_2 (\mu_t, \nu_t) \le 2 \sup \left\{ \left| \int_M f \, \mathrm{d} (\mu_t - \nu_t) \right| \, : \, f \in \dot{H}_1(\mu_t ) , \| f \|_{\dot{H}_1(\mu_t)} \le 1 \right\} . \]
Fix a function $f \in \dot{H}_1(\mu_t )$, $\| f \|_{\dot{H}_1(\mu_t)} \le 1$, and observe that
\[ \begin{split} \left| \int_M f \, \mathrm{d} (\mu_t - \nu_t) \right| &= \left| \frac{1}{\mathrm{Vol}(M)} \int_M \int_M P(t,x,y) f(y) \, \mathrm{d}(\mu - \nu )(x) \mathrm{d} \mathrm{Vol}(y) \right| \\ &= \left| \sum_{k=0}^{\infty} e^{-\lambda_k t} \frac{1}{\mathrm{Vol}(M)} \int_M \int_M \phi_k(x) \phi_k(y) f(y) \, \mathrm{d} (\mu - \nu )(x) \mathrm{d} \mathrm{Vol} (y) \right| \\ &= \left| \sum_{k=1}^{\infty} e^{-\lambda_k t} \widehat{f}(k) (\widehat{\mu}(k) - \widehat{\nu}(k)) \right| \\ &\le \left( \sum_{k=1}^{\infty} \lambda_k |\widehat{f}(k)|^2 \right)^{1/2} \left( \sum_{k=1}^{\infty} \frac{e^{-2 \lambda_k t}}{\lambda_k} |\widehat{\mu}(k) - \widehat{\nu}(k)|^2 \right)^{1/2} . \end{split} \]
It is easy to see that $\mu_t \ge c_2(\mu, 2t)^2 \mathrm{Vol} / \mathrm{Vol}(M)$ with $c_2 (\mu, t)>0$ as in \eqref{c1c2def}. Lemma \ref{sobolevlemma} thus shows that here
\[ \sum_{k=1}^{\infty} \lambda_k |\widehat{f}(k)|^2 \le \frac{1}{\mathrm{Vol}(M)} \int_M |\nabla f|^2 \, \mathrm{d} \mathrm{Vol} \le \frac{1}{c_2 (\mu, 2t)^2} \int_M |\nabla f|^2 \, \mathrm{d}\mu_t \le \frac{1}{c_2(\mu, 2t)^2} , \]
and we obtain
\[ W_2 (\mu_t, \nu_t) \le \frac{2}{c_2(\mu, 2t)} \left( \sum_{k=1}^{\infty} \frac{e^{-2 \lambda_k t}}{\lambda_k} |\widehat{\mu}(k) - \widehat{\nu}(k)|^2 \right)^{1/2} . \]
The estimates above for the terms in \eqref{triangleW2} give that for any real $t>0$,
\[ W_2 (\mu, \nu ) \le c_1(\mu) \left( 2dt + K(M) (2t)^{3/2} \right)^{1/2} + \frac{2}{c_2(\mu, 2t)} \left( \sum_{k=1}^{\infty} \frac{e^{-2 \lambda_k t}}{\lambda_k} |\widehat{\mu}(k) - \widehat{\nu}(k)|^2 \right)^{1/2} , \]
as claimed.
\end{proof}

\section{Weakly dependent random variables}\label{weaklydependentsection}

We now prove our main result on the empirical measure of weakly dependent random variables, and then derive Theorems \ref{alphamixingtheorem}, \ref{betamixingtheorem} and \ref{circletheorem}. Let $X_1, X_2, \dots, X_N$ be identically distributed $M$-valued random variables, each with distribution $\mu \in \mathcal{P}(M)$, and let $\mu_N=N^{-1} \sum_{n=1}^N \delta_{X_n}$.
\begin{thm}\label{weaklydependenttheorem} Assume that $\mu \ge c \mathrm{Vol} / \mathrm{Vol}(M)$ with some constant $c \ge 0$. For any real $t>0$,
\[ \sqrt{\mathbb{E} W_2^2( \mu_N , \mu )} \le c_1(\mu ) \left( d t+K(M)t^{3/2} \right)^{1/2} + \frac{2}{c_2(\mu, t)} \left( \frac{1}{N} E + \frac{2}{N^2} \sum_{1 \le m<n \le N} E_{m,n} \right)^{1/2} , \]
where $c_1(\mu)$, $c_2(\mu, t)$ and $K(M)$ are as in \eqref{c1c2def} and \eqref{K(M)}, and
\[ \begin{split} E &= \sum_{k=1}^{\infty} \frac{e^{- \lambda_k t}}{\lambda_k} \left( \int_{M} \phi_k^2 \, \mathrm{d} \mu - \left( \int_M \phi_k \mathrm{d} \mu \right)^2 \right) , \\ E_{m,n} &= \min \left\{ 2 \beta (X_m, X_n) \sup_{x,y \in M} \left| \sum_{k=1}^{\infty} \frac{e^{- \lambda_k t}}{\lambda_k} \phi_k (x) \phi_k (y) \right| , 4 \alpha (X_m,X_n) \sum_{k=1}^{\infty} \frac{e^{- \lambda_k t}}{\lambda_k} \sup_{x \in M} |\phi_k (x)|^2 \right\} . \end{split} \]
\end{thm}

\begin{proof} Applying Theorem \ref{smoothingtheorem} to $W_2 (\mu_N, \mu)$ and the triangle inequality for the $L^2$-norm leads to
\[ \sqrt{\mathbb{E}W_2^2 (\mu_N, \mu )} \le c_1 (\mu ) \left( dt+K(M)t^{3/2} \right)^{1/2} + \frac{2}{c_2(\mu, t)} \left( \mathbb{E} \sum_{k=1}^{\infty} \frac{e^{-\lambda_k t}}{\lambda_k} \left| \widehat{\mu_N}(k) - \widehat{\mu}(k) \right|^2 \right)^{1/2} . \]
Here
\[ \begin{split} \mathbb{E} \sum_{k=1}^{\infty} \frac{e^{-\lambda_k t}}{\lambda_k} \left| \widehat{\mu_N}(k) - \widehat{\mu}(k) \right|^2 = &\mathbb{E} \sum_{k=1}^{\infty} \frac{e^{- \lambda_k t}}{\lambda_k} \left( \frac{1}{N} \sum_{n=1}^N (\phi_k (X_n) -\widehat{\mu}(k)) \right)^2 \\ = &\frac{1}{N} \sum_{k=1}^{\infty} \frac{e^{- \lambda_k t}}{\lambda_k} \left( \int_{M} \phi_k^2 \, \mathrm{d} \mu - \left( \int_M \phi_k \mathrm{d} \mu \right)^2 \right) \\ &+\frac{2}{N^2} \sum_{1 \le m<n \le N} \mathbb{E} \sum_{k=1}^{\infty} \frac{e^{- \lambda_k t}}{\lambda_k} (\phi_k(X_m)-\widehat{\mu}(k))(\phi_k(X_n)-\widehat{\mu}(k)) . \end{split} \]
It remains to estimate the last line of the previous formula in two different ways: in terms of the $\beta$-mixing and the $\alpha$-mixing coefficients.

By the interpretation of the $\beta$-mixing coefficient as a total variation distance on $M \times M$, we  have
\[ \left| \mathbb{E} \sum_{k=1}^{\infty} \frac{e^{- \lambda_k }}{\lambda_k} (\phi_k(X_m)-\widehat{\mu}(k))(\phi_k(X_n)-\widehat{\mu}(k)) \right| \le 2 \beta (X_m, X_n) \sup_{x,y \in M} \left|  \sum_{k=1}^{\infty} \frac{e^{- \lambda_k t}}{\lambda_k} \phi_k(x) \phi_k(y) \right| . \]
On the other hand, using the $\alpha$-mixing coefficients to estimate the covariance of $\phi_k(X_m)$ and $\phi_k (X_n)$, we obtain
\[  \left| \mathbb{E} \sum_{k=1}^{\infty} \frac{e^{- \lambda_k t}}{\lambda_k} (\phi_k(X_m)-\widehat{\mu}(k))(\phi_k(X_n)-\widehat{\mu}(k)) \right| \le \sum_{k=1}^{\infty} \frac{e^{- \lambda_k t}}{\lambda_k} 4 \alpha (X_m, X_n) \sup_{x \in M} |\phi_k (x)|^2 . \]
The previous two estimates establish the formula for $E_{m,n}$.
\end{proof}

\begin{proof}[Proof of Theorem \ref{alphamixingtheorem}] Theorem \ref{weaklydependenttheorem} shows that for any $N \ge 1$ and any $t>0$,
\[ \sqrt{\mathbb{E}W_2^2(\mu_N, \mu )} \le C(M) t^{1/2} + \psi (t) \left( \frac{1}{N} + \frac{1}{N^2} \sum_{1 \le m < n \le N} \alpha (X_m, X_n) \right)^{1/2} \]
with some constant $C(M)>0$ and some function $\psi (t) >0$ depending only on the manifold (see the remark made on $c_2(\mu, t)$ after Theorem \ref{smoothingtheorem}). Letting $t \to 0$ slowly enough in terms of $N$, we deduce $\sqrt{\mathbb{E}W_2^2(\mu_N, \mu )} \to 0$ as $N \to \infty$, as claimed.
\end{proof}

\begin{proof}[Proof of Theorem \ref{betamixingtheorem}] In this proof $C(M)>0$ denotes a constant depending only on the manifold, whose value changes from line to line. The Minakshisundaram--Pleijel asymptotic expansion for the heat kernel \cite[p.\ 154]{CH} implies the diagonal estimate
\[ P(t,x,x) = \frac{\mathrm{Vol}(M)}{(4 \pi t)^{d/2}} \left( 1+ O(t) \right) \qquad \textrm{as } t \to 0^+ \]
with an implied constant depending only on $M$, as well as the off-diagonal estimate
\[ P(t,x,y) \le \frac{C(M)}{t^{d/2}} \qquad t>0, \,\, x,y \in M . \]
To apply Theorem \ref{weaklydependenttheorem}, observe that for all $0 <t \le 1$,
\[ \begin{split} \left| \sum_{k=1}^{\infty} \frac{e^{- \lambda_k t}}{\lambda_k} \phi_k (x) \phi_k (y) \right| &= \left| \int_{t}^{1} \sum_{k=1}^{\infty} e^{- \lambda_k u} \phi_k (x) \phi_k (y) \, \mathrm{d} u + \sum_{k=1}^{\infty} \frac{e^{- \lambda_k}}{\lambda_k} \phi_k (x) \phi_k (y) \right| \\ &\le \int_{t}^{1}  |P(u,x,y) -1| \, \mathrm{d} u +C(M). \end{split} \]
By the diagonal resp.\ off-diagonal heat kernel estimate above, $E$ resp.\ $E_{m,n}$ in Theorem \ref{weaklydependenttheorem} thus satisfy
\begin{equation}\label{Eestimate}
E \le \int_M \sum_{k=1}^{\infty} \frac{e^{- \lambda_k t}}{\lambda_k} |\phi_k (x) |^2 \, \mathrm{d} \mu (x) \le \int_{t}^{1} \left( \frac{\mathrm{Vol}(M)}{(4 \pi u)^{d/2}} + \frac{C(M)}{u^{d/2-1}} \right) \, \mathrm{d} u +C(M)
\end{equation}
and
\begin{equation}\label{Emnestimate}
E_{m,n} \le \beta (X_m, X_n) \left( \int_{t}^{1} \frac{C(M)}{u^{d/2}} \, \mathrm{d}u + C(M) \right) .
\end{equation}

First, assume that $d=2$. Then \eqref{Eestimate} and \eqref{Emnestimate} give
\[ \begin{split} E &\le \frac{\mathrm{Vol}(M)}{4 \pi} \log \frac{1}{t} +C(M), \\ E_{m,n} &\le \beta (X_m, X_n) C(M) \left( 1+\log \frac{1}{t} \right) . \end{split} \]
Hence Theorem \ref{weaklydependenttheorem} leads to
\[ \sqrt{\mathbb{E} W_2^2 (\mu_N, \mu )} \le C(M) t^{1/2} + \frac{2}{c^{1/2}N^{1/2}} \left( \frac{\mathrm{Vol}(M)}{4 \pi} \log \frac{1}{t} +C(M) + BC(M) \left( 1+ \log \frac{1}{t} \right) \right)^{1/2} . \]
Choosing $t=1/N$ proves claim (i).

Next, assume that $d \ge 3$. Then \eqref{Eestimate} and \eqref{Emnestimate} give
\[ \begin{split} E &\le \frac{\mathrm{Vol}(M)}{(4 \pi )^{d/2}} \cdot \frac{1}{(d/2-1)t^{d/2-1}} + C(M) R_d(t), \\ E_{m,n} &\le \beta (X_m, X_n) \frac{C(M)}{t^{d/2-1}} , \end{split} \]
where
\[ R_d(t)= \left\{ \begin{array}{ll} 1 & \textrm{if } d=3, \\ 1+ \log \frac{1}{t} & \textrm{if } d=4, \\ \frac{1}{t^{d/2-2}} & \textrm{if } d \ge 5 . \end{array} \right. \]
Hence Theorem \ref{weaklydependenttheorem} leads to
\[ \begin{split} \sqrt{\mathbb{E} W_2^2 (\mu_N, \mu )} \le &c_1(\mu ) (dt)^{1/2} + C(M) t^{3/4} \\ &+ \frac{2}{c^{1/2} N^{1/2}} \left( \frac{\mathrm{Vol}(M)}{(4 \pi )^{d/2}} \cdot \frac{1}{(d/2-1)t^{d/2-1}} +C(M) R_d(t) + \frac{B C(M)}{t^{d/2-1}} \right)^{1/2} . \end{split} \]
The optimal choice is
\[ t^{d/4} = \min \left\{ 1, \frac{2}{c_1 (\mu ) d^{1/2} c^{1/2} N^{1/2}} \left( \frac{\mathrm{Vol}(M)}{(4 \pi )^{d/2} (d/2-1)} + B C(M) \right)^{1/2} \right\} . \]
If $t=1$, then claim (ii) follows from the trivial estimate $W_2 (\mu_N, \mu) \le C(M)$. If $t<1$, then
\[ t^{3/4} \le C(M) \left( \frac{1+B}{c N} \right)^{3/(2d)} \qquad \textrm{and} \qquad \frac{R_d(t)^{1/2}}{c^{1/2}N^{1/2}} \le C(M) \left( \frac{1+B}{c N} \right)^{3/(2d)} , \]
and claim (ii) follows once again.

If there exists an orthonormal basis $\phi_k$, $k \ge 0$ of Laplace eigenfunctions such that $\sup_{k \ge 0} \sup_{x \in M} |\phi_k(x)| < \infty$, then we can use the second estimate for $E_{m,n}$ in Theorem \ref{weaklydependenttheorem} to show that \eqref{Emnestimate}, and consequently claims (i) and (ii) hold with $\beta (X_m,X_n)$ replaced by $\alpha (X_m,X_n)$.
\end{proof}

\begin{proof}[Proof of Theorem \ref{circletheorem}] The spectrum of $M= \mathbb{R}/\mathbb{Z}$ is the set of values $4 \pi^2 k^2$, $k \ge 0$. For each $k \neq 0$, the multiplicity is $2$, and the corresponding orthonormal eigenfunctions $\sqrt{2} \sin (2 \pi k x)$ and $\sqrt{2} \cos (2 \pi k x)$ have sup-norm $\sqrt{2}$. Hence $E$ and $E_{m,n}$ in Theorem \ref{weaklydependenttheorem} satisfy
\[ E \le 2 \sum_{k=1}^{\infty} \frac{e^{-4 \pi^2 k^2 t}}{4 \pi^2 k^2} 2 \le \sum_{k=1}^{\infty} \frac{1}{\pi^2 k^2} = \frac{1}{6}, \]
and similarly
\[ E_{m,n} \le 4 \alpha (X_m, X_n) 2 \sum_{k=1}^{\infty} \frac{e^{-4 \pi^2 k^2 t}}{4 \pi^2 k^2} 2 \le \frac{2}{3} \alpha (X_m, X_n) . \]
Theorem \ref{weaklydependenttheorem} thus gives that for any real $t>0$,
\[ \sqrt{\mathbb{E} W_2^2 (\mu_N, \mu)} \le c_1(\mu ) t^{1/2} + \frac{2}{c^{1/2}N^{1/2}} \left( \frac{1}{6} + \frac{4}{3} B \right)^{1/2} , \]
and the claim follows from letting $t \to 0$.
\end{proof}

\section{Random walks on compact Lie groups}

Let $G$ be a compact, connected Lie group of (real) dimension $d$. Fixing an Ad-invariant inner product on the Lie algebra gives rise to an invariant Riemannian metric on $G$, which we normalize so that $\mathrm{Vol}(G)=1$. The Wasserstein metric is defined in terms of the corresponding geodesic distance $\rho$ on $G$.

Let $\widehat{G}$ denote the unitary dual, and $d_{\pi}$ resp.\ $\lambda_{\pi}$ the dimension resp.\ Laplace eigenvalue of an irreducible unitary representation $\pi \in \widehat{G}$; that is, $\Delta \pi = -\lambda_{\pi} \pi$, where $\Delta$ acts entrywise. Let $\pi_0 \in \widehat{G}$ be the trivial representation. The Fourier coefficients of $f \in L^1 (G, \mathrm{Vol})$ and $\nu \in \mathcal{P}(G)$ are $\widehat{f}(\pi ) = \int_G f \pi^* \, \mathrm{d} \mathrm{Vol}$ and $\widehat{\nu}(\pi ) = \int_G \pi^* \, \mathrm{d} \nu$, $\pi \in \widehat{G}$. Throughout, $T^*$ denotes the adjoint of an operator $T$, and $\mathrm{SRad}(T)$, $\| T \|_{\mathrm{op}}$, $\| T \|_{\mathrm{HS}}=\sqrt{\mathrm{tr} (T^*T)}$ its spectral radius, operator norm, Hilbert--Schmidt norm, respectively.

Our main tool is still the Berry--Esseen inequality in Theorem \ref{smoothingtheorem}, now applied with $\mu = \mathrm{Vol}$, and hence $c_1(\mu )=c_2(\mu, t)=1$. Recall also that $K(G)=0$, as the Ricci curvature is positive semidefinite. The inequality now reads, for any $\nu \in \mathcal{P}(G)$ and any real $t>0$,
\begin{equation}\label{smoothingliegroup}
W_2 (\nu, \mathrm{Vol}) \le (dt)^{1/2} + 2 \left( \sum_{\substack{\pi \in \widehat{G} \\ \pi \neq \pi_0}} \frac{e^{-\lambda_{\pi} t}}{\lambda_{\pi}} d_{\pi} \| \widehat{\nu} (\pi ) \|_{\mathrm{HS}}^2 \right)^{1/2} .
\end{equation}

Indeed, we know that the entries $\{ \sqrt{d_{\pi}} \pi_{mn} \, : \, 1 \le m,n \le d_{\pi}, \,\, \pi \in \widehat{G} \}$ form an orthonormal basis of $L^2(G,\mathrm{Vol})$ of \textit{complex-valued} eigenfunctions of the Laplacian, and to apply Theorem \ref{smoothingtheorem} it remains to construct an orthonormal basis $\phi_k$, $k \ge 0$ of \textit{real-valued} eigenfunctions. If $\pi \in \widehat{G}$ and its contragredient (complex conjugate) $\overline{\pi}$ are unitarily inequivalent, i.e.\ $\pi, \overline{\pi} \in \widehat{G}$ and $\pi \neq \overline{\pi}$, then the $2 d_{\pi}^2$ functions
\[ \frac{\sqrt{d_{\pi}}}{\sqrt{2}} (\pi_{mn} + \overline{\pi}_{mn}) \quad \textrm{and} \quad \frac{\sqrt{d_{\pi}}}{\sqrt{2}} \cdot \frac{\pi_{mn} - \overline{\pi}_{mn}}{i}, \quad 1 \le m,n \le d_{\pi} \]
is an orthonormal system of real-valued eigenfunctions of the Laplacian with the same eigenvalue $\lambda_{\pi}$, spanning the same subspace as the entries of $\pi$ and $\overline{\pi}$. With the notation of Section \ref{notationsection}, the corresponding $2d_{\pi}^2$ Fourier coefficients in the sense of manifolds satisfy
\[ \begin{split} \sum_k |\widehat{\nu}(k)|^2 &= \sum_{m,n=1}^{d_{\pi}} \left| \int_G \frac{\sqrt{d_{\pi}}}{\sqrt{2}} (\pi_{mn} + \overline{\pi}_{mn}) \, \mathrm{d} \nu \right|^2 + \sum_{m,n=1}^{d_{\pi}} \left| \int_G \frac{\sqrt{d_{\pi}}}{\sqrt{2}} \cdot \frac{\pi_{mn} - \overline{\pi}_{mn}}{i} \, \mathrm{d} \nu \right|^2 \\ &= 2 d_{\pi} \sum_{m,n=1}^{d_{\pi}} \left| \int_G \pi_{mn} \, \mathrm{d} \nu \right|^2 \\ &= 2d_{\pi} \| \widehat{\nu} (\pi) \|_{\mathrm{HS}}^2 . \end{split} \]
If $\pi \in \widehat{G}$ and its contragredient $\overline{\pi}$ are unitarily equivalent, then the entries of $\pi$ and those of $\overline{\pi}$ span the same $d_{\pi}^2$-dimensional subspace, which is consequently closed under complex conjugation. We can thus choose an orthonormal basis of real-valued eigenfunctions of the Laplacian spanning the same subspace. With the notation of Section \ref{notationsection}, the corresponding $d_{\pi}^2$ Fourier coefficients in the sense of manifolds are easily seen to satisfy $\sum_k |\widehat{\nu}(k)|^2 = d_{\pi} \| \widehat{\nu} (\pi) \|_{\mathrm{HS}}^2$. In particular, the spectrum $\lambda_k$, $k \ge 0$ of $G$ consists of the values $\lambda_{\pi}$ repeated $d_{\pi}^2$ times, $\pi \in \widehat{G}$. This reduces \eqref{smoothingliegroup} to Theorem \ref{smoothingtheorem}.

\subsection{Spectral gaps}\label{spectralgapsection}

Let $L_0^2(G, \mathrm{Vol}) = \{ f \in L^2(G, \mathrm{Vol}) \, : \, \int_G f \, \mathrm{Vol} =0 \}$ denote the orthogonal complement of the $1$-dimensional subspace of constant functions. Given $\nu \in \mathcal{P}(G)$, let $T_{\nu}: L_0^2(G, \mathrm{Vol}) \to L_0^2 (G, \mathrm{Vol})$,
\[ (T_{\nu} f)(x) = \int_G f(xy) \, \mathrm{d} \nu (y) \]
denote the corresponding Markov operator. Note the identity $T_{\mu * \nu} = T_{\mu} T_{\nu}$, and in particular, $T_{\nu^{*n}} = T_{\nu}^n$, where $\mu * \nu$ denotes convolution. Let
\[ q(\nu) := \sqrt{ \mathrm{SRad}(T_{\nu}^* T_{\nu}) } \le \| T_{\nu} \|_{\mathrm{op}} . \]
We say that $\nu \in \mathcal{P}(G)$ has a spectral gap if $\mathrm{SRad}(T_{\nu})<1$. Since $\mathrm{SRad}(T_{\nu}) = \lim_{m \to \infty} \| T_{\nu}^m \|_{\mathrm{op}}^{1/m}$, the existence of a spectral gap implies $q(\nu^{*m})<1$ for some $m \ge 1$. In case $T_{\nu}$ is a normal operator, we have $q(\nu ) = \mathrm{SRad}(T_{\nu}) = \| T_{\nu} \|_{\mathrm{op}}$.

We first prove an estimate for $W_2 (\nu, \mathrm{Vol})$ in terms of $q(\nu)$. For the sake of completeness, we include $\mathbb{R}^d/\mathbb{Z}^d$, $d=1,2$, the only compact, connected Lie groups of dimension $d \le 2$.
\begin{thm}\label{qnutheorem} Let $\nu \in \mathcal{P}(G)$.
\begin{enumerate}
\item[(i)] If $G=\mathbb{R}/\mathbb{Z}$, then
\[ W_2 (\nu, \mathrm{Vol}) \le \frac{q(\nu)}{\sqrt{3}} . \]
\item[(ii)] If $G=\mathbb{R}^2/\mathbb{Z}^2$, then
\[ W_2 (\nu, \mathrm{Vol}) \le q(\nu) \sqrt{\frac{2}{\pi} \log \frac{1}{q(\nu)}} + 3 q(\nu ) . \]
\item[(iii)] If $d \ge 3$, then
\[ W_2 (\nu, \mathrm{Vol}) \le \left( \frac{8}{d(d-2)} \right)^{1/d} \frac{\sqrt{d}}{\sqrt{\pi}} q(\nu)^{2/d} + C(G) q(\nu)^{3/d} \]
with some constant $C(G)>0$ depending only on $G$.
\end{enumerate}
\end{thm}
\noindent Theorem \ref{qnutheorem} immediately yields an upper bound to $W_2 (\nu_1 * \nu_2 * \cdots * \nu_n, \mathrm{Vol})$, i.e.\ the rate of weak convergence of a random walk with independent steps, in terms of $q(\nu_1 * \nu_2 * \cdots * \nu_n) \le \| T_{\nu_1} \|_{\mathrm{op}} \| T_{\nu_2} \|_{\mathrm{op}} \cdots \| T_{\nu_n} \|_{\mathrm{op}}$. In particular, if $\sup_{n \ge 1} \| T_{\nu_n} \|_{\mathrm{op}} <1$, then $W_2 (\nu_1 * \nu_2 * \cdots * \nu_n, \mathrm{Vol}) \to 0$ exponentially fast. As for random walks with i.i.d.\ steps, we have $W_2 (\nu^{*n}, \mathrm{Vol}) \to 0$ exponentially fast whenever $\nu$ has a spectral gap (even without assuming normality of the Markov operator $T_{\nu}$).

We also mention an application to deterministic point sets with a spectral gap. Assume that $\nu \in \mathcal{P}(G)$ is supported on at most $N$ points, and satisfies $q(\nu ) \le B N^{-1/2}$ with some constant $B>0$; note that such point sets exist only in noncommutative groups, that is, in dimension $d \ge 3$. Theorem \ref{qnutheorem} then yields
\begin{equation}\label{deterministicpointsets}
W_2 (\nu, \mathrm{Vol}) \le \left( \frac{8}{d(d-2)} \right)^{1/d} \frac{\sqrt{d} B^{2/d}}{\sqrt{\pi} N^{1/d}} + \frac{C(G)B^{3/d}}{N^{3/(2d)}} ,
\end{equation}
which shows that $\nu$ achieves the optimal distance $\ll N^{-1/d}$ from the uniform distribution, see \eqref{quantization}. For instance, the classical explicit construction of Lubotzky, Phillips and Sarnak \cite{LPS1,LPS2} is a symmetric point set of size $N$ with equal weights on $\mathrm{SU}(2)$ and $\mathrm{SO}(3)$ for which $q(\nu) = 2\sqrt{N-1}/N$, a value which is in fact smallest possible among all symmetric sets of the same size. Note that a symmetric set with equal weights corresponds to a self-adjoint Markov operator $T_{\nu}$. Observe also that \eqref{deterministicpointsets} and \eqref{quantization} together imply that any $\nu \in \mathcal{P}(G)$ supported on at most $N$ points (not necessarily symmetric or with equal weights) satisfies $q(\nu) \gg N^{-1/2}$, in accordance with general results of Kesten on spectral properties of random walks \cite{KE}.

We now return to random walks, and prove a nonasymptotic upper estimate for the distance of the empirical measure from the uniform distribution under a spectral gap condition. Let $Y_1, Y_2, \dots, Y_N$ be independent $G$-valued random variables with distribution $\nu_1, \nu_2, \dots, \nu_N \in \mathcal{P}(G)$, respectively. Let $S_n=Y_1 Y_2 \cdots Y_n$ be the corresponding random walk, and let $\mu_N = N^{-1} \sum_{n=1}^N \delta_{S_n}$ be its empirical measure.
\begin{thm}\label{rwempiricaltheorem} Assume that $\sum_{1 \le m<n \le N} q(\nu_{m+1} * \nu_{m+2} * \cdots * \nu_n ) \le BN$ with some constant $B \ge 0$.
\begin{enumerate}
\item[(i)] If $G=\mathbb{R}/\mathbb{Z}$, then
\[ \sqrt{\mathbb{E} W_2^2 (\mu_N, \mathrm{Vol})} \le \frac{\sqrt{1+2B}}{\sqrt{3N}} . \]
\item[(ii)] If $G=\mathbb{R}^2/\mathbb{Z}^2$, then
\[ \sqrt{\mathbb{E} W_2^2 (\mu_N, \mathrm{Vol})} \le \sqrt{\frac{1+2B}{\pi}} \cdot \sqrt{\frac{\log N}{N}} + \frac{3 \sqrt{1+2B}}{\sqrt{N}} . \]
\item[(iii)] If $d \ge 3$, then
\[ \sqrt{\mathbb{E} W_2^2 (\mu_N, \mathrm{Vol})} \le \left( \frac{8+16 B}{d(d-2)} \right)^{1/d} \frac{\sqrt{d}}{\sqrt{\pi}N^{1/d}}  + \frac{C(G)(1+B)^{3/(2d)}}{N^{3/(2d)}} \]
with some constant $C(G)>0$ depending only on $G$.
\end{enumerate}
\end{thm}
For instance, if $\sup_{2 \le n \le N} \| T_{\nu_n} \|_{\mathrm{op}} \le p<1$, then Theorem \ref{rwempiricaltheorem} applies with $B=p/(1-p)$. Random walks with i.i.d.\ steps having a spectral gap attain optimal rate, even without assuming normality of the Markov operator.

We conclude this section with the proof of Theorems \ref{qnutheorem} and \ref{rwempiricaltheorem}.
\begin{proof}[Proof of Theorem \ref{qnutheorem}] Observe that the $d_{\pi}^2$-dimensional subspace in $L_0^2 (G, \mathrm{Vol})$ spanned by the orthonormal system $\{ \sqrt{d_{\pi}} \pi_{mn} \, : \, 1 \le m,n \le d_{\pi} \}$ is invariant under $T_{\nu}$, and that the restriction of $T_{\nu}$ to this subspace has Hilbert--Schmidt norm square
\[ \sum_{m,n=1}^{d_{\pi}} \| T_{\nu} \sqrt{d_{\pi}} \pi_{mn} \|_{L^2(G, \mathrm{Vol})}^2 = d_{\pi} \| \widehat{\nu} (\pi) \|_{\mathrm{HS}}^2 . \]
Therefore $d_{\pi} \| \widehat{\nu}(\pi) \|_{\mathrm{HS}}^2 \le d_{\pi}^2 q(\nu )^2$, and the Berry--Esseen inequality \eqref{smoothingliegroup} implies that for all real $t>0$,
\begin{equation}\label{smoothingqnu}
W_2 (\nu, \mathrm{Vol}) \le (dt)^{1/2} + 2q(\nu ) \left( \sum_{\substack{\pi \in \widehat{G} \\ \pi \neq \pi_0}} \frac{e^{-\lambda_{\pi}t}}{\lambda_{\pi}} d_{\pi}^2 \right)^{1/2} = (dt)^{1/2} + 2q(\nu ) \left( \sum_{k=1}^{\infty} \frac{e^{-\lambda_k t}}{\lambda_k} \right)^{1/2} .
\end{equation}
For the sake of simplicity, in the second step we expressed the infinite series in terms of the spectrum of $G$ as a manifold.

Assume first, that $G=\mathbb{R}/\mathbb{Z}$. Then \eqref{smoothingqnu} yields
\[ \begin{split} W_2 (\nu, \mathrm{Vol}) \le t^{1/2} + 2q(\nu ) \left( \sum_{\substack{k \in \mathbb{Z} \\ k \neq 0}} \frac{e^{-4 \pi^2 k^2 t}}{4 \pi^2 k^2} \right)^{1/2} &\le t^{1/2} + 2q(\nu ) \left( \sum_{\substack{k \in \mathbb{Z} \\ k \neq 0}} \frac{1}{4 \pi^2 k^2} \right)^{1/2} \\ &= t^{1/2} + 2 q(\nu ) \left( \frac{1}{12} \right)^{1/2} , \end{split} \]
and claim (i) follows from letting $t \to 0$.

Next, assume that $G=\mathbb{R}^2 / \mathbb{Z}^2$. Then \eqref{smoothingqnu} yields
\begin{equation}\label{smoothingqnu2dim}
W_2 (\nu, \mathrm{Vol}) \le (2t)^{1/2} + 2q(\nu ) \left( \sum_{\substack{k \in \mathbb{Z}^2 \\ k \neq 0}} \frac{e^{-4 \pi^2 |k|^2 t}}{4 \pi^2 |k|^2} \right)^{1/2} .
\end{equation}
To estimate the infinite series, let $N(x)=|\{ k \in \mathbb{Z}^2 \, : \, 0<|k|^2< x \}|$ denote the number of lattice points other than the origin in the open disk centered at the origin of radius $\sqrt{x}$. Estimating $N(x)$ is the famous Gauss circle problem, but we shall only need the trivial upper bound
\[ N(x) \le \pi \left( \sqrt{x} + \frac{1}{\sqrt{2}} \right)^2 -1 \le \pi x + \left( \pi \sqrt{2} + \frac{\pi}{2} -1 \right) \sqrt{x} \qquad ( x \ge 1) . \]
This follows from drawing unit squares centered at lattice points $k$ with $0<|k|^2<x$, and noting that the union of these pairwise disjoint unit squares is a subset of the disk centered at the origin of radius $\sqrt{x}+1/\sqrt{2}$. Note that $N(x)=0$ for $0 \le x \le 1$. Elementary calculations show that for all $0<t \le 1$,
\[ \begin{split} \sum_{\substack{k \in \mathbb{Z}^2 \\ k \neq 0}} \frac{e^{-4 \pi^2 |k|^2 t}}{4 \pi^2 |k|^2} &= \int_0^{\infty} \frac{e^{-4 \pi^2 t x}}{4 \pi^2 x} \, \mathrm{d}N(x) \\ &= - \int_0^{\infty} N(x) \, \mathrm{d} \frac{e^{-4 \pi^2 tx}}{4 \pi^2 x} \\ &\le \int_1^{\infty} \left(  \pi x + \left( \pi \sqrt{2} + \frac{\pi}{2} -1 \right) \sqrt{x} \right) \left( \frac{t}{x} + \frac{1}{4 \pi^2 x^2} \right) e^{-4 \pi^2 t x} \, \mathrm{d}x \\ &\le \int_1^{\infty} \frac{1}{4 \pi x} e^{-4 \pi^2 t x} \, \mathrm{d}x + \frac{3 \pi \sqrt{2} + \frac{5}{2} \pi - 3}{4 \pi^2} \\ &\le \int_{1}^{1/t} \frac{1}{4 \pi x} \, \mathrm{d} x + \int_{1/t}^{\infty} \frac{t}{4 \pi} e^{-4 \pi^2 t x} \, \mathrm{d} x +\frac{3 \pi \sqrt{2} + \frac{5}{2} \pi - 3}{4 \pi^2} \\ &= \frac{1}{4 \pi} \log \frac{1}{t} + \frac{e^{-4 \pi^2}}{16 \pi^3} + \frac{3 \pi \sqrt{2} + \frac{5}{2} \pi - 3}{4 \pi^2} . \end{split} \]
Letting $\tau$ denote the constant term in the last line of the previous formula, \eqref{smoothingqnu2dim} thus gives
\[ W_2 (\nu, \mathrm{Vol}) \le (2t)^{1/2} + 2 q(\nu) \left( \frac{1}{4 \pi} \log \frac{1}{t} + \tau \right)^{1/2} . \]
Choose $t=q(\nu)^2 \le 1$, and note that $2^{1/2} + 2 \tau^{1/2} \approx 2.77$. This proves claim (ii).

Finally, assume that $d \ge 3$. Similarly to the proof of Theorem \ref{betamixingtheorem}, the Minakshisun\-daram--Pleijel asymptotic expansion of the heat kernel now gives
\[ \sum_{k=0}^{\infty} e^{-\lambda_k t} = \frac{1}{(4 \pi t)^{d/2}} \left( 1 +O(t) \right) \qquad \textrm{as } t \to 0^+ . \]
Consequently for any $0<t \le 1$,
\[ \begin{split} \sum_{k=1}^{\infty} \frac{e^{-\lambda_k t}}{\lambda_k} = \int_{t}^{1} \sum_{k=1}^{\infty} e^{-\lambda_k u} \, \mathrm{d} u + \sum_{k=1}^{\infty} \frac{e^{-\lambda_k}}{\lambda_k} &\le \int_{t}^{1} \left( \frac{1}{(4 \pi u)^{d/2}} + \frac{C(G)}{u^{d/2-1}} \right) \, \mathrm{d} u + C(G) \\ &\le \frac{1}{(4 \pi)^{d/2}(d/2-1)t^{d/2-1}} + C(G) R_d(t) , \end{split} \]
where
\[ R_d(t)= \left\{ \begin{array}{ll} 1 & \textrm{if } d=3, \\ 1+ \log \frac{1}{t} & \textrm{if } d=4, \\ \frac{1}{t^{d/2-2}} & \textrm{if } d \ge 5 . \end{array} \right. \]
Hence \eqref{smoothingqnu} yields
\[ W_2 (\nu, \mathrm{Vol}) \le (dt)^{1/2} + 2q(\nu) \left( \frac{1}{(4 \pi)^{d/2}(d/2-1)t^{d/2-1}} + C(G) R_d(t) \right)^{1/2} . \]
The optimal choice is
\[ t^{d/2} = \min \left\{ 1, 4 q(\nu)^2 \frac{1}{(4 \pi)^{d/2} d (d/2-1)} \right\} . \]
If $t=1$, then claim (iii) follows from the trivial estimate $W_2 (\nu, \mathrm{Vol}) \le C(G)$. If $t<1$, then $q(\nu) R_d (t)^{1/2} \le C(G) q(\nu)^{3/d} $, and claim (iii) follows once again.
\end{proof}

\begin{proof}[Proof of Theorem \ref{rwempiricaltheorem}] Applying the Berry--Esseen inequality \eqref{smoothingliegroup} and the triangle inequality for the $L^2$-norm gives that for any real $t>0$,
\[ \sqrt{\mathbb{E} W_2^2 (\mu_N, \mathrm{Vol})} \le (dt)^{1/2} + 2 \left( \sum_{\substack{\pi \in \widehat{G} \\ \pi \neq \pi_0}} \frac{e^{-\lambda_{\pi} t}}{\lambda_{\pi}} d_{\pi} \mathbb{E} \| \widehat{\mu_N}(\pi) \|_{\mathrm{HS}}^2 \right)^{1/2} . \]
Here
\[ \mathbb{E} \| \widehat{\mu_N}(\pi) \|_{\mathrm{HS}}^2 = \mathbb{E} \left\| \frac{1}{N} \sum_{n=1}^N \pi (S_n) \right\|_{\mathrm{HS}}^2 = \frac{1}{N^2} \sum_{m,n=1}^N \mathbb{E} \, \mathrm{tr} \left( \pi (S_m)^* \pi (S_n) \right) . \]
Since $\pi(x)$ is a $d_{\pi} \times d_{\pi}$ unitary matrix, the total contribution of the diagonal terms $m=n$ is $d_{\pi}/N$. If $m<n$, then
\[ \begin{split} \left| \mathbb{E} \, \mathrm{tr} (\pi (S_m)^* \pi (S_n)) \right| &= \left| \mathbb{E} \, \mathrm{tr} \left( \pi(Y_m)^* \pi (Y_{m-1})^* \cdots \pi (Y_1)^* \pi (Y_1) \cdots \pi (Y_{n-1}) \pi (Y_n) \right) \right| \\ &= \left| \mathbb{E} \, \mathrm{tr} \, \pi ( Y_{m+1} Y_{m+2} \cdots Y_n ) \right| \\ &= \left| \mathrm{tr} (\nu_{m+1} * \nu_{m+2} * \cdots * \nu_n) \,\, \widehat{} \,\, (\pi) \right| \\ &\le \sqrt{d_{\pi}} \| (\nu_{m+1} * \nu_{m+2} * \cdots * \nu_n) \,\, \widehat{} \,\, (\pi) \|_{\mathrm{HS}} \\ &\le d_{\pi} q (\nu_{m+1} * \nu_{m+2} * \cdots * \nu_n ) .  \end{split} \]
In the last two steps we used the Cauchy--Schwarz inequality, and an observation from the proof of Theorem \ref{qnutheorem}. Estimating the terms $m>n$ is entirely analogous, and we obtain
\[ \frac{1}{N^2} \sum_{m,n=1}^N \mathbb{E} \, \mathrm{tr} \left( \pi (S_m)^* \pi (S_n) \right) \le \frac{d_{\pi}}{N} + \frac{2d_{\pi}}{N^2} \sum_{1 \le m<n \le N} q(\nu_{m+1} * \nu_{m+2} * \cdots * \nu_n) \le \frac{d_{\pi}}{N} (1+2B) . \]
Hence for any real $t>0$,
\[ \sqrt{\mathbb{E} W_2^2 (\mu_N, \mathrm{Vol})} \le (dt)^{1/2} + \frac{2(1+2B)^{1/2}}{N^{1/2}} \left( \sum_{\substack{\pi \in \widehat{G} \\ \pi \neq \pi_0}} \frac{e^{-\lambda_{\pi} t}}{\lambda_{\pi}} d_{\pi}^2 \right)^{1/2} . \]
The rest of the proof is identical to that of Theorem \ref{qnutheorem} with $q$ replaced by $(1+2B)^{1/2}/N^{1/2}$.
\end{proof}

\subsection{Nonuniform spectral gaps}

Let $G$ be a semisimple, compact, connected Lie group. The proof of Theorem \ref{semisimpletheorem} is based on a result of Varj\'u \cite{VA}, who proved that for any $\nu \in \mathcal{P}(G)$ and any $r>0$,
\begin{equation}\label{varju}
1- \max_{\substack{\pi \in \widehat{G} \\ 0< \lambda_{\pi} \le r}} \| \widehat{\nu}(\pi) \|_{\mathrm{op}} \ge c_0 \left( 1- \max_{\substack{\pi \in \widehat{G} \\ 0< \lambda_{\pi} \le r_0}} \| \widehat{\nu}(\pi) \|_{\mathrm{op}} \right) \frac{1}{\log^2 (r+2)} ,
\end{equation}
where $c_0,r_0>0$ are constants depending only on $G$. Varj\'u calls this a nonuniform spectral gap estimate, since the lower bound to the gap depends on the size $r$ of the Laplace eigenvalues. In fact, he proved \eqref{varju} with a factor $1/\log^{\gamma}(r+2)$, where $1 \le \gamma \le 2$ is a constant depending explicitly on the group; for the sake of simplicity, we formulated Theorem \ref{semisimpletheorem} in the worst case $\gamma =2$. The spectral gap conjecture basically states that the same holds with $\gamma =0$.

\begin{proof}[Proof of Theorem \ref{semisimpletheorem}] Let $c_0, r_0>0$ be as in \eqref{varju}, and let $\nu \in \mathcal{P}(G)$ be such that $\nu^{*n} \to \mathrm{Vol}$ weakly. For any $\pi \in \widehat{G}$, $\pi \neq \pi_0$ we then have $\widehat{\nu^{*n}}(\pi) = \widehat{\nu}(\pi)^n \to 0$. Consequently $\mathrm{SRad}(\widehat{\nu}(\pi)) <1$, and there exists an integer $m_0=m_0(\nu )$ such that $\| \widehat{\nu}(\pi )^{m_0} \|_{\mathrm{op}} <1$ for all of the finitely many $\pi \in \widehat{G}$ with $0<\lambda_{\pi} \le r_0$. We now apply \eqref{varju} to $\nu^{*m_0}$, and obtain that for any integer $n \ge 1$ and any $r>0$,
\begin{equation}\label{nonunifgap}
\begin{split} \max_{\substack{\pi \in \widehat{G} \\ 0< \lambda_{\pi} \le r}} \| \widehat{\nu}(\pi)^n \|_{\mathrm{op}} \le \left( \max_{\substack{\pi \in \widehat{G} \\ 0< \lambda_{\pi} \le r}} \| \widehat{\nu}(\pi)^{m_0} \|_{\mathrm{op}} \right)^{\lfloor n/m_0 \rfloor} &\le \left( 1- \frac{bm_0}{\log^2 (r+2)} \right)^{n/m_0-1} \\ &\le 2 \exp \left( -\frac{bn}{\log^2 (r+2)} \right) \end{split}
\end{equation}
with the constant
\[ b=b(\nu ) := \frac{c_0}{m_0} \left( 1- \max_{\substack{\pi \in \widehat{G} \\ 0< \lambda_{\pi} \le r_0}} \| \widehat{\nu}(\pi)^{m_0} \|_{\mathrm{op}} \right) >0. \]

First, we estimate $W_2 (\nu^{*n}, \mathrm{Vol})$. The nonuniform spectral gap estimate \eqref{nonunifgap} yields
\[ d_{\pi} \| \widehat{\nu}(\pi)^n \|_{\mathrm{HS}}^2 \le d_{\pi}^2 \| \widehat{\nu}(\pi)^n \|_{\mathrm{op}}^2 \le 4 d_{\pi}^2 \exp \left( -\frac{bn}{\log^2 (\lambda_{\pi}+2)} \right) . \]
Hence the Berry--Esseen inequality \eqref{smoothingliegroup} gives that for any real $t>0$,
\[ \begin{split} W_2 (\nu^{*n}, \mathrm{Vol}) &\ll t^{1/2} + \left( \sum_{\substack{\pi \in \widehat{G} \\ \pi \neq \pi_0}} \frac{e^{-\lambda_{\pi} t}}{\lambda_{\pi}} d_{\pi}^2 \exp \left( -\frac{bn}{\log^2 (\lambda_{\pi}+2)} \right) \right)^{1/2} \\ &= t^{1/2} + \left( \sum_{k=1}^{\infty} \frac{e^{-\lambda_k t}}{\lambda_k} \exp \left( -\frac{bn}{\log^2 (\lambda_k+2)} \right) \right)^{1/2} . \end{split} \]
For the sake of simplicity, in the second step we expressed the infinite series in terms of the spectrum of $G$ as a manifold. Using Weyl's law $|\{ k \in \mathbb{N} \, : \, \lambda_k \le x  \}| \sim \kappa_d x^{d/2}$ as $x \to \infty$ with some (explicit) constant $\kappa_d>0$, it is straightforward to check that for any $x>0$ with $xt$ large enough, $\sum_{\lambda_k > x} e^{-\lambda_k t} \ll x^{d/2} e^{-xt}$. Estimating the terms $\lambda_k \le x$ and $\lambda_k >x$ separately, we thus deduce
\[ \begin{split} \sum_{k=1}^{\infty} \frac{e^{-\lambda_k t}}{\lambda_k} \exp \left( -\frac{bn}{\log^2 (\lambda_k+2)} \right) &\ll \sum_{\lambda_k \le x} \exp \left( - \frac{bn}{\log^2 (x+2)} \right) + \sum_{\lambda_k >x} e^{-\lambda_k t} \\ &\ll x^{d/2} \exp \left( -\frac{bn}{\log^2 (x+2)} \right) + x^{d/2} e^{-xt}.  \end{split} \]
One readily checks that the optimal choice is $x=\exp (a_0 n^{1/3})$ and $t=n^{1/3} \exp (-a_0 n^{1/3})$ with a suitably small constant $a_0=a_0(\nu)>0$, in which case we obtain
\[ W_2 (\nu^{*n}, \mathrm{Vol}) \ll t^{1/2} + \exp (-a_0 n^{1/3}) \ll \exp \left( -(a_0/4) n^{1/3} \right) , \]
as claimed.

Next, we estimate $\sqrt{\mathbb{E} W_2^2 (\mu_N, \mathrm{Vol})}$. The Berry--Esseen inequality \eqref{smoothingliegroup} and the triangle inequality for the $L^2$-norm give that for any real $t>0$,
\[ \sqrt{\mathbb{E} W_2^2 (\mu_N, \mathrm{Vol})} \le (dt)^{1/2} + 2 \left( \sum_{\substack{\pi \in \widehat{G} \\ \pi \neq \pi_0}} \frac{e^{-\lambda_{\pi} t}}{\lambda_{\pi}} d_{\pi} \mathbb{E} \| \widehat{\mu_N}(\pi) \|_{\mathrm{HS}}^2 \right)^{1/2} . \]
Following the steps in the proof of Theorem \ref{rwempiricaltheorem}, here
\[ \mathbb{E} \| \widehat{\mu_N}(\pi) \|_{\mathrm{HS}}^2 \le \frac{d_{\pi}}{N} + \frac{2}{N^2} \sum_{1 \le m<n \le N} \left| \mathrm{tr} \left( \widehat{\nu}(\pi)^{n-m} \right) \right| \le \frac{d_{\pi}}{N} + \frac{2 d_{\pi}}{N^2} \sum_{1 \le m<n \le N} \| \widehat{\nu}(\pi)^{n-m} \|_{\mathrm{op}} . \]
The nonuniform spectral gap estimate \eqref{nonunifgap} shows that
\[ \sum_{1 \le m<n \le N} \| \widehat{\nu}(\pi)^{n-m} \|_{\mathrm{op}} \le 2 \sum_{1 \le m<n \le N} \exp \left( -\frac{b (n-m)}{\log^2 (\lambda_{\pi} +2)} \right) \ll N \log^2 (\lambda_{\pi} +2) , \]
therefore
\[ \begin{split} \sqrt{\mathbb{E} W_2^2 (\mu_N, \mathrm{Vol})} &\ll t^{1/2} + \left( \sum_{\substack{\pi \in \widehat{G} \\ \pi \neq \pi_0}} \frac{e^{-\lambda_{\pi} t}}{\lambda_{\pi}} d_{\pi}^2 \frac{\log^2 (\lambda_{\pi} +2)}{N} \right)^{1/2} \\ &= t^{1/2} + \frac{1}{N^{1/2}} \left( \sum_{k=1}^{\infty} \frac{e^{-\lambda_k t}}{\lambda_k} \log^2 (\lambda_k +2) \right)^{1/2} \end{split} \]
with an implied constant depending only on $\nu$ and $G$. From Weyl's law we deduce
\[ \sum_{k=1}^{\infty} \frac{e^{-\lambda_k t}}{\lambda_k} \log^2 (\lambda_k +2) \le \sum_{\lambda_k \ll 1/t} \frac{\log^2 (\lambda_k +2)}{\lambda_k} + \sum_{\lambda_k \gg 1/t} e^{-\lambda_k t} t \log^2 \frac{1}{t} \ll t^{1-d/2} \log^2 \frac{1}{t} ,\]
hence
\[  \sqrt{\mathbb{E} W_2^2 (\mu_N, \mathrm{Vol})} \ll t^{1/2} + \frac{1}{N^{1/2}} t^{1/2-d/4} \log \frac{1}{t} . \]
The optimal choice is $t=N^{-2/d} (\log N)^{4/d}$, and the claim follows.
\end{proof}

\subsection*{Acknowledgments}

The author is supported by the Austrian Science Fund (FWF) projects F-5510 and Y-901. I would like to thank the two anonymous referees for useful comments and suggestions.


\begin{thebibliography}{99}
\footnotesize{

\bibitem{AKT} M.\ Ajtai, J.\ Koml\'os and G.\ Tusn\'ady: \textit{On optimal matchings.} Combinatorica 4 (1984), 259--264.

\bibitem{AST} L.\ Ambrosio, F.\ Stra and D.\ Trevisan: \textit{A PDE approach to a 2-dimensional matching problem.} Probab.\ Theory Related Fields 173 (2019), 433--477.

\bibitem{AU} T.\ Aubin: \textit{Nonlinear Analysis on Manifolds. Monge--Amp\`ere Equations.} Grundlehren der Mathematischen Wissenschaften, vol.\ 252. Springer-Verlag, New York, 1982.

\bibitem{BW} F.\ Bach and J.\ Weed: \textit{Sharp asymptotic and finite-sample rates of convergence of empirical measures in Wasserstein distance.} Bernoulli 25 (2019), 2620--2648.

\bibitem{BS} Y.\ Benoist and N.\ de Saxc\'e: \textit{A spectral gap theorem in simple Lie groups.} Invent.\ Math.\ 205 (2016), 337--361.

\bibitem{BL1} S.\ Bobkov and M.\ Ledoux: \textit{A simple Fourier analytic proof of the AKT optimal matching theorem.} Ann.\ Appl.\ Probab.\ 31 (2021), 2567--2584.

\bibitem{BL2} S.\ Bobkov and M.\ Ledoux: \textit{One-dimensional empirical measures, order statistics, and Kantorovich transport distances.} Mem.\ Amer.\ Math.\ Soc.\ 261 (2019), no.\ 1259.

\bibitem{BL3} S.\ Bobkov and M.\ Ledoux: \textit{Transport inequalities on Euclidean spaces for non-Euclidean metrics.} J.\ Fourier Anal.\ Appl.\ 26 (2020), Paper No.\ 60.

\bibitem{BOI} E.\ Boissard: \textit{Simple bounds for convergence of empirical and occupation measures in 1-Wasserstein distance.} Electron.\ J.\ Probab.\ 16 (2011), 2296--2333.

\bibitem{BLG} E.\ Boissard and T.\ Le Gouic: \textit{On the mean speed of convergence of empirical and occupation measures in Wasserstein distance.} Ann.\ Inst.\ Henri Poincar\'e Probab.\ Stat.\ 50 (2014), 539--563.

\bibitem{BOR1} B.\ Borda: \textit{Berry--Esseen smoothing inequality for the Wasserstein metric on compact Lie groups.} J.\ Fourier Anal.\ Appl.\ 27 (2021), Paper No.\ 13.

\bibitem{BOR2} B.\ Borda: \textit{Equidistribution of random walks on compact groups II. The Wasserstein metric.} Bernoulli 27 (2021), 2598--2623.

\bibitem{BOU} N.\ Bourbaki: \textit{Lie Groups and Lie Algebras. Chapters 7--9.} Translated from the 1975 and 1982 French originals by Andrew Pressley. Elements of Mathematics (Berlin). Springer-Verlag, Berlin, 2005.

\bibitem{BG1} J.\ Bourgain and A.\ Gamburd: \textit{On the spectral gap for finitely-generated subgroups of $\mathrm{SU}(2)$.} Invent.\ Math.\ 171 (2008), 83--121.

\bibitem{BG2} J.\ Bourgain and A.\ Gamburd: \textit{A spectral gap theorem in $\mathrm{SU}(d)$.} J.\ Eur.\ Math.\ Soc.\ 14 (2012), 1455--1511.

\bibitem{BR} R.\ Bradley: \textit{Basic properties of strong mixing conditions. A survey and some open questions.} Update of, and a supplement to, the 1986 original. Probab.\ Surv.\ 2 (2005), 107--144.

\bibitem{BST} L.\ Brown and S.\ Steinerberger: \textit{On the Wasserstein distance between classical sequences and the Lebesgue measure.} Trans.\ Amer.\ Math.\ Soc.\ 373 (2020), 8943--8962.

\bibitem{CA} S.\ Caracciolo, C.\ Lucibello, G.\ Parisi and G.\ Sicuro: \textit{Scaling hypothesis for the Euclidean bipartite matching problem.} Phys.\ Rev.\ E 90 (2014), 012118.

\bibitem{CH} I.\ Chavel: \textit{Eigenvalues in Riemannian Geometry.} Pure and Applied Mathematics, 115. Academic Press, Inc., Orlando, FL, 1984.

\bibitem{CL} L.\ Clozel: \textit{Automorphic forms and the distribution of points on odd-dimensional spheres.} Israel J.\ Math.\ 132 (2002), 175--187.

\bibitem{DM} J.\ Dedecker and F.\ Merlev\`ede: \textit{Behavior of the Wasserstein distance between the empirical and the marginal distributions of stationary $\alpha$-dependent sequences.} Bernoulli 23 (2017), 2083--2127.

\bibitem{FG} N.\ Fournier and A.\ Guillin: \textit{On the rate of convergence in Wasserstein distance of the empirical measure.} Probab.\ Theory Related Fields 162 (2015), 707--738.

\bibitem{GL} S.\ Graf and H.\ Luschgy: \textit{Foundations of Quantization for Probability Distributions.} Lecture Notes in Mathematics, 1730. Springer-Verlag, Berlin, 2000.

\bibitem{KI} Y.\ Kawada and K.\ It\^o: \textit{On the probability distribution on a compact group. I.} Proc.\ Phys.-Math.\ Soc.\ Japan 22 (1940), 977--998.

\bibitem{KE} H.\ Kesten: \textit{Symmetric random walks on groups.} Trans.\ Amer.\ Math.\ Soc.\ 92 (1959), 336--354.

\bibitem{KL} B.\ Kloeckner: \textit{Approximation by finitely supported measures.} ESAIM Control Optim.\ Calc.\ Var.\ 18 (2012), 343--359.

\bibitem{LPS1} A.\ Lubotzky, R.\ Phillips and P.\ Sarnak: \textit{Hecke operators and distributing points on the sphere. I.} Comm.\ Pure Appl.\ Math.\ 39 (1986), 149--186.

\bibitem{LPS2} A.\ Lubotzky, R.\ Phillips and P.\ Sarnak: \textit{Hecke operators and distributing points on $S^2$. II.} Comm.\ Pure Appl.\ Math.\ 40 (1987), 401--420.

\bibitem{MA} H.\ Masuda: \textit{Ergodicity and exponential $\beta$-mixing bounds for multidimensional diffusions with jumps.} Stochastic Process.\ Appl.\ 117 (2007), 35--56.

\bibitem{MP} M.\ Meitz and P.\ Saikkonen: \textit{Subgeometric ergodicity and $\beta$-mixing.} J.\ Appl.\ Prob.\ 58 (2021), 594--608.

\bibitem{OH} H.\ Oh: \textit{The Ruziewicz problem and distributing points on homogeneous spaces of a compact Lie group.} Israel J.\ Math.\ 149 (2005), 301--316.

\bibitem{PEY} R.\ Peyre: \textit{Comparison between $W_2$ distance and $\dot{H}^{-1}$ norm, and localization of Wasserstein distance.} ESAIM Control Optim.\ Calc.\ Var.\ 24 (2018), 1489--1501.

\bibitem{STE1} S.\ Steinerberger: \textit{A Wasserstein inequality and minimal Green energy on compact manifolds.} J.\ Funct.\ Anal.\ 281 (2021), Paper No.\ 109076.

\bibitem{STE2} S.\ Steinerberger: \textit{Wasserstein distance, Fourier series and applications.} Monatsh.\ Math.\ 194 (2021), 305--338.

\bibitem{STR} K.\ Stromberg: \textit{Probabilities on a compact group.} Trans.\ Amer.\ Math.\ Soc.\ 94 (1960), 295--309.

\bibitem{TA} M.\ Talagrand: \textit{Matching random samples in many dimensions.} Ann.\ Appl.\ Probab.\ 2 (1992), 846--856.

\bibitem{TZ} J.\ Toth and S.\ Zelditch: \textit{Riemannian manifolds with uniformly bounded eigenfunctions.} Duke Math.\ J.\ 111 (2002), 97--132.

\bibitem{VK} J.\ VanderKam: \textit{$L^{\infty}$ norms and quantum ergodicity on the sphere.} Internat.\ Math.\ Res.\ Notices 1997 (1997), 329--347.

\bibitem{VA} P.\ Varj\'u: \textit{Random walks in compact groups.} Doc.\ Math.\ 18 (2013), 1137--1175.

\bibitem{VE} A.\ Veretennikov: \textit{On polynomial mixing bounds for stochastic differential equations.} Stochastic Process.\ Appl.\ 70 (1997), 115--127.

\bibitem{VI} C.\ Villani: \textit{Topics in Optimal Transportation.} Graduate Studies in Mathematics, 58. American Mathematical Society, Providence, RI, 2003.

\bibitem{VR1} V.\ Volkonskii and Yu.\ Rozanov: \textit{Some limit theorems for random functions I.} Theor.\ Probab.\ Appl.\ 4 (1959), 178--197.

\bibitem{VR2} V.\ Volkonskii and Yu.\ Rozanov: \textit{Some limit theorems for random functions II.} Theor.\ Probab.\ Appl.\ 6 (1961), 186--198.

\bibitem{WE} G.\ Wei: \textit{Manifolds with a lower Ricci curvature bound.} Surveys in differential geometry. Vol.\ XI, 203--227, Surv.\ Differ.\ Geom., 11, Int.\ Press, Sommerville, MA, 2007.

}
\end{thebibliography}
\end{document}